\documentclass[preprint,12pt]{elsarticle}

\usepackage{amsmath, amssymb, amsthm}
\usepackage{graphicx}
\usepackage{hyperref}

\newtheorem{theorem}{Theorem}[section]
\newtheorem{lemma}[theorem]{Lemma}
\newtheorem{proposition}[theorem]{Proposition}
\newtheorem{corollary}[theorem]{Corollary}

\theoremstyle{definition}
\newtheorem{definition}[theorem]{Definition}

\theoremstyle{remark}

\numberwithin{equation}{section}

\begin{document}
%================================================

\begin{frontmatter}

\title{Commuting Graph of Unitriangular Group \(UT(4,p)\)}

\author[1,2]{Miseriya Majeed\corref{cor1}}
\cortext[cor1]{Corresponding author}
\ead{miseriyamajeed@gmail.com}   % email of corresponding author
\author[1,2]{Ramkumar P.B}
\ead{ramkumar_pb@rajagiritech.edu.in}

\affiliation[1]{
organization={Rajagiri School of Engineering and Technology},
city={Cochin},
country={India}
}

\affiliation[2]{
organization={APJ Abdul Kalam Technological University},
city={Thiruvananthapuram},
country={India}
}

\begin{abstract}
Let \(G=UT(4,p)\) be the group of all \(4\times 4\) unitriangular matrices over the finite field \(\mathbb F_p\), where \(p\) is a prime. Using the six-parameter form of the elements of \(G\), we describe the commutativity relation explicitly and use it to analyse the structure of the graph. We prove that the reduced commuting graph is connected and has diameter \(3\). We also determine the size of its maximal cliques, chromatic number, independence number, perfectness, etc. by decomposing the graph into cosets, layers, and direction parts.
\end{abstract}

\begin{keyword}
Commuting graph  \sep Unitriangular group \sep Clique number \sep Chromatic number

\MSC[2020] 05C25 \sep 20D15 \sep 05C69
\end{keyword}

\end{frontmatter}
\section{Introduction}

The commuting graph of a finite group is a natural way to study commutativity by graph-theoretic methods. Over the years, this graph has been studied for many families of finite groups, and graph-theoretic properties such as connectedness, diameter, clique number, and perfectness have been studied. Iranmanesh and Jafarzadeh studied the commuting graphs of symmetric and alternating groups and showed that, the graph is either disconnected or has bounded diameter.\cite{IranmaneshJafarzadeh2008}. In a different direction, Giudici and Pope investigated commuting graphs of linear groups and matrix rings over the integers modulo \(m\), and proved connectedness and diameter results \cite{GiudiciPope2010}. Later, Morgan and Parker showed that for a finite group with trivial centre, every connected component of the commuting graph has diameter at most \(10\) \cite{MorganParker2013}. These results show that diameter questions are central in the theory.

Beyond connectedness and diameter, commuting graphs have also been studied through their structural graph-theoretic properties. Britnell and Gill studied perfect commuting graphs and classified the finite quasisimple groups whose commuting graphs are perfect \cite{BritnellGill2017}. More recently, Arvind, Cameron, Ma and Maslova considered broader aspects of commuting graphs, including recognition questions and graph classes such as cographs and chordal graphs \cite{ArvindCameronMaMaslova2023}. 
There are also related works in matrix and ring settings. Akbari, Bidkhori and Mohammadian studied commuting graphs of matrix algebras \cite{AkbariBidkhoriMohammadian2008}, while Abdollahi considered commuting graphs of full matrix rings over finite fields \cite{Abdollahi2008}. 

Abdollahi, Akbari and Maimani introduced the non-commuting graph of a group and studied how graph-theoretic properties of this graph are related to the structure of the group \cite{AbdollahiAkbariMaimani2006}. Darafsheh, Ghorbani and Prajapati studied maximal subsets of pairwise non-commuting elements in finite \(p\)-groups \cite{DarafshehGhorbaniPrajapati2015}, while Azad, Iranmanesh, Praeger and Spiga considered similar questions for finite general linear groups through abelian coverings and non-commuting sets \cite{AzadIranmaneshPraegerSpiga2010}.

For unitriangular groups, however, the literature is more closely tied to subgroup structure and automorphisms than to a detailed study of the commuting graph in small dimension. In particular, Mahalanobis studied the automorphism group of the group of unitriangular matrices over a field and emphasized the role of maximal abelian normal subgroups in understanding its structure \cite{Mahalanobis2013}. This is especially relevant here, since large cliques in the commuting graph arise naturally from abelian subgroups. A closely related work is that of Kumar and Prajapati on maximal non-commuting sets in certain unipotent upper-triangular linear groups \cite{KumarPrajapati2017}. Their work is relevant to the present paper because maximal non-commuting subsets are naturally connected with independence problems in commuting graphs. They decompose the relevant vertices into non-commuting parts and use this structure to study the corresponding independence number. In our setting, the subset \(\mathcal L_0^{*}\) corresponds to the union \(N_2\cup N_3\cup N_3^{\mathrm{anti}}\) in their notation. Thus the problem is closely related, although the method used here, namely the coset and layer decomposition, is different and yields a different bound.

In this paper we study the reduced commuting graph of \(UT(4,p)\). The case \(UT(4,p)\) is small enough that the commutation relations can be written down explicitly, but at the same time it is rich enough to exhibit nontrivial graph-theoretic behaviour. Our approach is based on an explicit decomposition of the graph using the coordinate description of elements of \(UT(4,p)\). This leads naturally to cosets, layers, direction parts, and giant cliques, and these pieces together reveal much of the global structure of the graph.

The paper is organized as follows. Section~2 contains the preliminaries and the basic commutativity relations in \(UT(4,p)\). Section~3 establishes the basic graph-theoretic properties of the reduced commuting graph, including connectivity, diameter, and clique number. Section~4 studies maximal abelian subgroups and uses them to describe maximal cliques. Section~5 develops the coset and layer decomposition, and analyses the interaction between direction parts and zero-layers. Section~6 is devoted to the chromatic number and the independence number, with special attention to the zero-layer union \(\mathcal L_0^*\) and its clique decomposition.
\section{Preliminaries}
\begin{definition}
    Let $G$ be a finite non--abelian group. The commuting graph of $G$, denoted by $\Gamma(G)$, is the graph whose vertex set is $G$, in which two distinct vertices $x$ and $y$ are adjacent if and only if $xy = yx$.
\end{definition}
For any element $x\in G$, the centralizer of $x$ in $G$ is
$C_G(x)=\{g\in G \mid gx=xg\}$.
In this paper, we also consider the reduced commuting graph, denoted by $\Gamma_{\mathrm{red}}(G)$, which is obtained by removing the center $Z(G)$ in $\Gamma(G)$.
Let $\Gamma=(V,E)$ be a graph. The minimum and maximum degrees of $\Gamma(G)$ is denoted by $\delta(\Gamma(G))$ and $\Delta(\Gamma(G))$ respectively. The diameter, $diam(G)$ of a connected graph is the maximum distance between two vertices. A clique is a subset of vertices such that every pair of vertices in the subset is adjacent. The clique number of a graph $\Gamma(G)$, denoted by $\omega(\Gamma(G))$, is the size of the largest clique in $\Gamma(G)$. An independent set is a set of vertices, no two of which are adjacent. The independence number of $\Gamma(G)$, denoted by $\alpha(\Gamma(G))$, is the maximum size of an independent set. The chromatic number $\chi(\Gamma(G))$ is the minimum number of colors required to color the vertices so that the adjacent vertices receive different colors. A graph \(\Gamma(G)\) is called \emph{perfect} if, for every induced subgraph
\(\Lambda\) of \(\Gamma(G)\), the chromatic number of \(\Lambda\) is equal to its
clique number; that is,
\(
\chi(\Lambda)=\omega(\Lambda).
\) These parameters are studied for the reduced commuting graph of $UT(4,p)$ in the subsequent sections.
\begin{definition}
     Let $p$ be a prime and let $\mathbb{F}_p$ denote the finite field with $p$ elements. The unitriangular group $UT(4,p)$ is the group of all $4 \times 4$ upper triangular matrices over $\mathbb{F}_p$ with all diagonal entries equal to $1$.
\end{definition}
Throughout this paper, we write $G=UT(4,p)$ and every element of $UT(4,p)$ can be written in the form
\[
\begin{pmatrix}
1 & a & b & c \\
0 & 1 & d & e \\
0 & 0 & 1 & f \\
0 & 0 & 0 & 1
\end{pmatrix}
\]
where $a,b,c,d,e,f \in \mathbb{F}_p$.
For convenience we identify this matrix with the $6$--tuple $(a,b,c,d,e,f)$. For $x=(a,b,c,d,e,f)$ and $y=(A,B,C,D,E,F),$
The group multiplication is given by
\[
xy
=
(a + A,\; b + B + aD,\; c + C + aE + bF,\; d + D,\; e + E + dF,\; f + F).
\]
Thus $|UT(4,p)| = p^6$.
The center of a group $G$ is defined by
$
Z(G)=\{x\in G \mid xy=yx \text{ for all } y\in G\}.
$ For the group $UT(4,p)$,
$
Z(G)=\{(0,0,c,0,0,0)\mid c\in \mathbb{F}_p\},
$and therefore $|Z(G)|=p$.
The commuting graph encodes the commutativity relations among elements of the group. Also,
$xy = yx$ if and only if
\begin{equation}{\label{commuting1}}
aD = Ad,
\end{equation}

\begin{equation}{\label{commuting2}}
dF = Df,
\end{equation}

\begin{equation}{\label{commuting3}}
aE + bF = Ae + Bf.
\end{equation}
These are the three commuting equations that determine adjacency in the commuting graph of $UT(4,p)$.

\textbf{Centralizers:} Solving the commuting equations for a general element of \(UT(4,p)\) gives the possible orders of centralizers listed in Table \ref{tab:centralizer_conditions}.
\begin{table}[h]
\centering
\begin{tabular}{|c|l|}
\hline
$|C_G(x)|$ & Condition on the element $x=(a,b,c,d,e,f)$ \\ \hline

$p^6$ & $a=b=d=e=f=0$ \\ \hline

$p^5$ & $a=d=f=0,\ (b,e)\neq(0,0)$ \\ \hline

$p^3$ & $d\neq0,\ (a,f)\neq(0,0)$ \\ \hline

$p^4$ & otherwise, 1.\ $d\neq0,\ (a,f)=(0,0)$ \\
      &2. $d=0,\ (a,f)\neq(0,0)$ \\ \hline

\end{tabular}
\caption{Possible orders of the centralizer $C_G(x)$ in $UT(4,p)$.}
\label{tab:centralizer_conditions}
\end{table}
These values follow directly from equations \eqref{commuting1}, \eqref{commuting2} and \eqref{commuting3}. The first row corresponds to the central elements of \(G\). The remaining rows are obtained by solving the three linear conditions imposed on a general element \((A,B,C,D,E,F)\in G\) which commutes with \(x=(a,b,c,d,e,f)\).
\begin{lemma}
    $\delta(\Gamma(G))=p^3-1$ and $\Delta(\Gamma(G))=p^6-1$
\end{lemma}
\begin{lemma}{\label{clique and subgroup}}{\cite{ArvindCameronMaMaslova2023}}
 A vertex subset $S$ of $G$ is a maximal clique of $\Gamma(G)$ if and only if $S$ is a maximal abelian subgroup of $G$. 
 
\end{lemma}
%================================================
\section{Structural Properties}
In this section, we prove some initial structural results for the reduced commuting graph of $UT(4,p)$. These include connectivity, diameter, and related clique-theoretic properties.
\begin{proposition}\label{connected}
$\Gamma_{red}(G)$ is connected
\end{proposition}
\begin{proof}
Let \(H=\{(0,b,c,d,e,0)\;:\; b,c,d,e \in \mathbb{F}_p \}\)
We have to prove that every non-central element $x=(a,b,c,d,e,f)$ commutes with some non-central element of $H$. Let \(h=(0,B,C,D,E,0)\) $\in$ \(H\).\\ \textbf{Case 1:} if \(x \in H\) then \(x\) commutes with \(H\), since \(H\) is Abelian. \\
\textbf{Case 2:}if $x \notin H$, then $(a,f) \neq (0,0)$. From $aD = 0$ and $Df = 0$, and since at least one of $a,f$ is nonzero, we obtain $D = 0$. The remaining condition is $aE = Bf$.
If $a \neq 0$, choose any $B \in \mathbb{F}_p$ and set $E = \frac{f}{a}B$.
If $a = 0$, then $f \neq 0$. The equation becomes $0 = Bf$, so take $B = 0$ and choose any $E \in \mathbb{F}_p$.
In both cases $C$ is arbitrary. Thus there exists $h \in H$ such that $xh = hx$. Therefore $x$ is adjacent to some element of $H$.
\end{proof}

\begin{proposition}\label{diameter}
\(
\operatorname{diam}(\Gamma_{\mathrm{red}}(G))\le 3.
\)
\end{proposition}
\begin{proof}
For every \(x,y\in \Gamma_{\mathrm{red}}(G)\), by Proposition \ref{connected} there exist \(h_1,h_2\in H\) such that \(x\sim h_1\) and \(y\sim h_2\). Since \(H\) is abelian, \(h_1\sim h_2\). Hence
\(
x\sim h_1\sim h_2\sim y
\)
is a path of length at most \(3\). Therefore
\(
\operatorname{diam}(\Gamma_{\mathrm{red}})\le 3.
\)
\end{proof}
\begin{theorem}
\(
diam(\Gamma_{\text{red}}(G))=3
\)
\end{theorem}

\begin{proof}
Let $x=(0,0,0,0,0,1)$ and $y=(1,0,0,1,0,0)$. 
A direct computation shows that they do not commute. Moreover,
$
C_G(x)=\{(A,0,C,0,E,F):A,C,E,F\in\mathbb{F}_p\}
$ and $C_G(y)=\{(A,B,C,A,0,0):A,B,C\in\mathbb{F}_p\}$. So,
$C_G(x)\cap C_G(y)=Z(G)$ in $\Gamma(G).$
This implies that the only common neighbors of $x$ and $y$ in $\Gamma(G)$ are the central elements. Hence the distance between $x$ and $y$ in $\Gamma_{red}(G)$ is at least $3$. Combining with Proposition \ref{diameter} we obtain 
$diam(\Gamma_{\text{red}}(G))=3$
\end{proof}
While the following determination of the clique number is likely known, in view of its connection with maximal abelian subgroups, we record it here for completeness and to make the paper self-LevchukSuleimanova2012. Related background on maximal abelian normal subgroups of maximal unipotent subgroups may be found in \cite{LevchukSuleimanova2012}, which in particular applies to the type $A_n$ case containing $UT(4,p)$.
\begin{lemma}{\label{p4}}
Every Abelian subgroup $S \leq G$ satisfies $|S| \leq p^4$.
\end{lemma}
\begin{proof}
Define a  group homomorphism
\[
\gamma : G \to \mathbb{F}_p^{3}, \text{ such that } \gamma(x) = (a,d,f)
\].
Let $S \leq G$ be abelian. Take $x,y \in S$ with
\[
\gamma(x) = (a,d,f) \text{ and }\gamma(y) = (A,D,F).
\]
Since $xy = yx$, Equation \eqref{commuting1} and \eqref{commuting2} imply that $(A,D,F)$ is a scalar multiple of $(a,d,f)$. Equivalently, all vectors $\gamma(h)$ with $h \in S$ lie in a single $1$-dimensional subspace of $\mathbb{F}_p^{3}$. Hence
\(
|\gamma(S)| \leq p .
\) Also,
\(
\ker(\gamma)=\{(0,b,c,0,e,0)\}
\)
has size $p^{3}$. Therefore
\begin{align*}
|S| &= |S \cap \ker(\gamma)| \cdot |\gamma(S)| \\
    &\leq |\ker(\gamma)| \cdot p \\
    &= p^{3} \cdot p = p^{4}\\
\end{align*} Thus $|S| \leq p^{4}$.
\end{proof}
\begin{proposition}
 $\omega(\Gamma(G))=p^{4}$.
\end{proposition}
\begin{proof}
There exists an abelian subgroup
\(
H=\{(0,b,c,d,e,0): b,c,d,e\in \mathbb{F}_p\}
\)
having order $p^{4}$. From Lemma \ref{p4}, there exist no abelian subgroup of order greater than $p^4$. So together with Lemma \ref{clique and subgroup}, shows maximum clique has size $p^{4}$.
\end{proof}
\section{Maximal Cliques}
In this section, Using the relation between cliques and abelian subgroups, we show that every maximal clique has order either $p^3$ or $p^4$. We first prove the following lemma.
\begin{lemma}{\label{lemmac}}
    Every maximal abelian subgroup of $UT(4,p)$ has size either $p^4$ or $p^3$
\end{lemma}
To prove this lemma, we first introduce the projection $\pi$ and establish the following three lemmas.     Let \(\pi : G \to \mathbb{F}_p^{\,2}\) be defined by
\(
\pi(a,b,c,d,e,f)=(a,f).
\) Then, $
\ker(\pi)=H.
$
Let \(A\le G\) be maximal abelian with \(A\not\subseteq H\). Then \(\pi(A)\neq \{(0,0)\}\), so
\(
|\pi(A)|\in\{p,p^2\}.
\)
Also,
\begin{equation}{\label{eqnc}}
  |A|=|A\cap H|\;|\pi(A)|.  
\end{equation}
\begin{lemma}{\label{lemma0}}
For any $x \in G\setminus H$, the set $H \cap C_G(x)$ has order $p^2$.
\end{lemma}

\begin{proof}
Let \(x=(a,b,c,d,e,f)\in G\setminus H\), so \((a,f)\neq(0,0)\), and let \(h=(0,B,C,D,E,0)\in H\). The commutativity conditions reduce to
\(
D=0,~ aE=Bf.
\)
Thus \((B,E)\) satisfies one linear relation, so it has exactly \(p\) solutions in \(\mathbb F_p^2\), while \(C\) is arbitrary with \(p\) choices. Therefore
\(
|H\cap C_G(x)|=p^2.
\)
\end{proof}
\begin{lemma}{\label{lemma1}}
If $|\pi(A)|=p$, then $|A\cap H|=p^2$.
\end{lemma}

\begin{proof}
Pick $x\in A\setminus H$. Since $|\pi(A)|=p$, the set $\pi(A)$ is a one-dimensional subspace of $\mathbb{F}_p^2$ generated by $\pi(x)$. So for every element $y\in A$, there exists $t\in\mathbb{F}_p$ such that
\(
\pi(y)=t\,\pi(x).
\)
Hence
\(
\pi(yx^{-t})=\pi(y)-t\,\pi(x)
=0.
\)
Thus
$yx^{-t}\in \ker\pi = H$ .
Also $yx^{-t}\in A$, hence
$yx^{-t}\in A\cap H$. Hence every $y\in A$ can be written as
\begin{equation}\label{eq:decomp}
y = x^t k ,~ k\in A\cap H .
\end{equation}
Now,
\begin{equation}\label{eqna}
A\cap H \subseteq H\cap C_G(x)
\end{equation}
since $A$ is abelian. By Lemma \ref{lemma0},
$|A\cap H|\le p^2$ .
Assume for contradiction that
\begin{equation}{\label{eqnb}}
|A\cap H|<p^2 .
\end{equation}
By \eqref{eqna}, \eqref{eqnb} and lemma \ref{lemma0}, there exists
\begin{equation}\label{eq:hchoice}
h\in (H\cap C_G(x))\setminus (A\cap H).
\end{equation}
Since $h\in C_G(x)$ by construction, we have $hx = xh$. Hence $h$ commutes with $x^t$ for all $t$.
Now $h\in H$ and $H$ is abelian, so $h$ commutes with every element of $A\cap H$ and together with the decomposition \eqref{eq:decomp} shows $h$ commutes with every $y\in A$.
Therefore $\langle A,h\rangle$ is an abelian subgroup strictly containing $A$ (because $h\notin A$) by \eqref{eq:hchoice}.
This contradicts the assumption that $A$ is maximal abelian. Hence our assumption \eqref{eqnb} was false. Therefore
\(
|A\cap H|=p^2 .
\)

\end{proof}
\begin{lemma}{\label{lemma2}}
If $|\pi(A)|=p^2$, then $|A\cap H|=p$.
\end{lemma}
\begin{proof}
If $|\pi(A)|=p^2$, then $\pi(A)=\mathbb{F}_p^2$. Hence $A$ contains elements $x_1$ and $x_2$ with
\(
\pi(x_1)=(1,0),~\pi(x_2)=(0,1).
\)
Let
\(
h=(0,B,C,D,E,0)\in A\cap H.
\) Since $A$ is abelian, $h$ must commute with both $x_1$ and $x_2$.
Commuting with $x_1$ forces $D=0$ and $E=0$, while commuting with $x_2$ forces $D=0$ and $B=0$. Hence
\(
h=(0,0,C,0,0,0)\in Z(G).
\)
Therefore
\(
A\cap H \subseteq Z(G).
\) Since $A$ is maximal abelian, it contains $Z(G)$. Hence 
\(
Z(G) \subseteq A \cap H
\)
always holds, we get
\(
|A \cap H| = |Z(G)| = p .
\)
\end{proof}
\begin{proof}[\textbf{Proof of Lemma \ref{lemmac}}]
Let \(A \leq G\) be a maximal abelian subgroup. If \(A \subseteq H\), then, since \(H\) is itself an abelian subgroup, the maximality of \(A\) implies that \(A=H\). Consequently, \(|A|=|H|=p^{4}\). On the other hand, if \(A \nsubseteq H\), then the following arguments show that \(|A|=p^{3}\). That is, if $|\pi(A)|=p$, then Equation \eqref{eqnc} and Lemma \ref{lemma1},
\(
|A|=p^2\cdot p=p^3.
\)
Similarly, If $|\pi(A)|=p^2$, then by Lemma \ref{lemma2}, 
\(
|A|=p\cdot p^2=p^3
\)
\end{proof}
\begin{theorem}
Every maximal clique of \(\Gamma(G)\) has order either \(p^{3}\) or \(p^{4}\).
\end{theorem}
\begin{proof}
    The result follows from Lemma \ref{clique and subgroup} and Lemma \ref{lemmac}.
\end{proof}
\section{Layer decomposition using cosets}
Consider the quotient \(G/H\), where
\(
H=\{(0,B,C,D,E,0):B,C,D,E\in \mathbb F_p\}.
\)
Let \(\mathcal C\) denote the collection of all nonzero cosets of \(H\) in \(G\). For each nonzero pair \((a,f)\in \mathbb F_p^2\), define
\(
C_{(a,f)}=(a,0,0,0,0,f)H.
\)
Thus
\(
\mathcal C=\{C_{(a,f)}:(a,f)\in \mathbb F_p^2\setminus\{(0,0)\}\}.
\)
Then the following theorem holds.
\begin{theorem}\leavevmode
\begin{enumerate}
\item Every nonzero coset \(C_{(a,f)}\) splits into \(p\) disjoint \(d\)-layers $L_d(a,f) = \{(a,b,c,d,e,f):b,c,e \in \mathbb{F}_p\}$.

\item Elements from different \(d\)-layers \(L_d(a,f)\) never commute.

\item Inside each layer \(L_d(a,f)\), commuting is determined by
\(
a(E-e)=f(B-b).
\)

\item \(|L_d(a,f)| = p^3.\)

\item Each element of \(L_d(a,f)\) commutes with exactly \(p^2\) elements of \(L_d(a,f)\).
\end{enumerate}

\end{theorem}

\begin{proof}
Since \(d\) takes \(p\) values, the coset \(C_{(a,f)}\) decomposes into \(p\) disjoint \(d\)-layers. If \(x\in L_d(a,f)\) and \(y\in L_D(a,f)\) commute, then the first two commuting equations force
\(
a(D-d)=0~ \text{and} ~f(D-d)=0.
\)
As \((a,f)\neq(0,0)\), it follows that \(D=d\). Hence different \(d\)-layers are pairwise noncommuting. Inside a fixed layer, the remaining commuting condition reduces to \(a(E-e)=f(B-b)\), which proves (3). The size of \(L_d(a,f)\) is immediate from the free choice of \(b,c,e\). Finally, for fixed \(x\in L_d(a,f)\), this criterion is a single linear relation in \(B\) and \(E\), while \(C\) remains arbitrary, so \(x\) commutes with exactly \(p^2\) elements of \(L_d(a,f)\).
\end{proof}
\begin{theorem}
For every nonzero coset \(C_{(a,f)}\) and every \(d\in \Gamma_p\), the induced subgraph on the \(d\)-layer \(L_d(a,f)\) satisfies
\(
\Gamma[L_d(a,f)] \cong pK_{p^2}.
\)
\end{theorem}
\begin{proof}
Let $x,y \in L_d(a,f)$ such that
\(
x=(a,b,c,d,e,f)  \text{and}  y=(a,B,C,d,E,f).
\) Then $x \sim y$ if and only if
\(
a(E-e)=f(B-b),
\) which is equivalent to
\(
aE-fB = ae-fb.
\)
Define a linear map
\(
\phi:\mathbb{F}_p^2 \to \mathbb{F}_p
\) by
\(
\phi(b,e)=ae-fb.
\)
Then
\(
x \sim y \iff \phi(B,E)=\phi(b,e).
\)
For each $t\in \mathbb{F}_p$, define
\[
L_{d,t}(a,f)=\{(a,b,c,d,e,f)\in C_{(a,f)}:\phi(b,e)=t\}.
\]

Since $\phi$ is linear, the sets $\phi^{-1}(t)$ partition $\mathbb{F}_p^2$ into $p$ affine lines, each of size $p$. Hence, each set $L_{d,t}(a,f)$ has size $p^2$.
Moreover, if $x,y\in L_{d,t}(a,f)$ then $\phi(b,e)=\phi(B,E)$, so $x$ and $y$ commute. Thus each $L_{d,t}(a,f)$ induces a clique of size $p^2$. If $x\in L_{d,s}(a,f)$ and $y\in L_{d,t}(a,f)$ with $s\ne t$, then $\phi(b,e)\ne \phi(B,E)$, so $x$ and $y$ do not commute.
\end{proof}

Take two cosets $C_{(a,f)}$ and $C_{(A,F)}$. Let $x=(a,b,c,d,e,f)\in C_{(a,f)}$ and $y=(A,B,C,D,E,F)\in C_{(A,F)}$. From commuting equations \eqref{commuting1} and \eqref{commuting2}, if $d$ and $D$ are not both zero, then $(a,f)$ and $(A,F)$ are proportional. Therefore, if $(a,f) $ and $(A,F)$ are not proportional, then equations \eqref{commuting1} and \eqref{commuting2} force  $d=D=0$.
Hence, \text{when} $C_{(a,f)}$ and $C_{(A,F)}$ are linearly independent, all the edges between  $ C_{(a,f)}$ and  $C_{(A,F)}$ are between  $L_0(a,f)$ and  $L_0(A,F)$. This leads to the following theorem.
\begin{theorem}{\label{L0}}
    Let \((a,f)\) and \((A,F)\) be linearly independent in \(\mathbb{F}_p^2\). Then commuting between vertices of \(C_{(a,f)}\) and \(C_{(A,F)}\) occurs only between the layers \(L_0(a,f)\) and \(L_0(A,F)\).
\end{theorem}
\begin{theorem}
If \((A,F)=k(a,f)\) for some \(k\in\mathbb F_p^\times\), Then the bipartite graph induced by the commuting graph between \(C_{(a,f)}\) and \(C_{(A,F)}\) satisfies
\(
\Gamma[C_{(a,f)},C_{(A,F)}]\cong p^{2}K_{p^{2},p^{2}} .
\)
\end{theorem}

\begin{proof}
Let \(x=(a,b,c,d,e,f)\in C_{(a,f)}\) and \(y=(ka,B,C,D,E,kf)\in C_{(ka,kf)}\). Then substituting  \(A=ka\) and \(F=kf\)  in \eqref{commuting1}, \eqref{commuting2} and \eqref{commuting3}, we obtain
\(
aD=kad, ~dkf=Df, ~aE+kbf=kae+Bf.
\)
Since \((a,f)\neq(0,0)\), the first two equations imply
\(
D=kd.
\)
Hence \(x\) and \(y\) commute if and only if
\(
D=kd\quad\text{and}\quad aE-fB=k(ae-fb).
\)
For \(d,t\in\mathbb F_p\), define
\[
L_{d,t}(a,f)=\{(a,b,c,d,e,f)\in C_{(a,f)}:ae-fb=t\},
\]
and
\[
L_{kd,kt}(ka,kf)=\{(ka,B,C,kd,E,kf)\in C_{(ka,kf)}:aE-fB=kt\}.
\]
Each set \(L_{d,t}(a,f)\) (and similarly \(L_{kd,kt}(ka,kf)\)) has size \(p^{2}\), since the equation \(ae-fb=t\) is one linear relation in \(b,e\) and the coordinate \(c\) is free.
Let \(x\in L_{d,t}(a,f)\) and \(y\in L_{d',t'}(ka,kf)\).  
From the commuting conditions derived above, \(x\) and \(y\) commute precisely when
\(
d'=kd\quad\text{and}\quad t'=kt.
\)
Thus \(L_{d,t}(a,f)\) is completely joined to \(L_{kd,kt}(ka,kf)\), and there are no edges to any other subsets.
Consequently, each pair
\(
L_{d,t}(a,f),~ L_{kd,kt}(ka,kf)
\)
forms a complete bipartite graph \(K_{p^{2},p^{2}}\). Since \((d,t)\) ranges over \(\mathbb F_p^2\), there are \(p^{2}\) such pairs, and they are mutually disjoint. Therefore the bipartite graph between \(C_{(a,f)}\) and \(C_{(A,F)}\) is the disjoint union of \(p^{2}\) copies of \(K_{p^{2},p^{2}}\), that is,
\(
\Gamma[C_{(a,f)},C_{(A,F)}]\cong p^{2}K_{p^{2},p^{2}}.
\)
\end{proof}
\begin{definition}
Two cosets $C_{(a,f)}$ and $C_{(A,F)}$ are said to have the same direction if
\(
(A,F)=\lambda(a,f)
\) for some $\lambda \in \mathbb{F}_p^{*}$.
\end{definition}

Thus the directions correspond to the one-dimensional subspaces of the vector space $\mathbb{F}_p^{2}$. For each nonzero direction $\ell=\langle (a,f)\rangle$, define the direction part

\[
U_\ell=\bigcup_{\lambda\in\mathbb{F}_p^{*}} C_{(\lambda a,\lambda f)} .
\]
Thus $U_\ell$ consists of all cosets whose parameters are scalar multiples of $(a,f)$.

Since every nonzero vector of $\mathbb{F}_p^2$ lies in exactly one one-dimensional subspace, the family \(
\{U_\ell:\ell\subset \mathbb{F}_p^2,\ \dim(\ell)=1\}
\) forms a partition of the set $G\setminus H$. Moreover, since the number of one-dimensional subspaces of $\mathbb{F}_p^2$ is $p+1$, there are exactly $p+1$ direction parts. Each direction part $U_\ell$ consists of $p-1$ cosets of $H$, namely
\(
C_{(a,f)},\ C_{(2a,2f)},\ \ldots,\ C_{((p-1)a,(p-1)f)}.
\)
Also define
\[
G_\ell(d,t)=\bigcup_{\lambda\in\mathbb F_p^\times} L_{\lambda d,\lambda t}(\lambda a,\lambda f).
\]
\begin{theorem}{\label{giant_clique}}
For any nonzero direction $\ell=\langle (a,f)\rangle,$ The induced subgraph $\Gamma[U_\ell]$ satisfies, 
\(
\Gamma[U_\ell]\cong p^2K_{(p-1)p^2}.
\)
\end{theorem}

\begin{proof}\leavevmode
We first show that the sets \(G_\ell(d,t)\), where \(d,t\in\mathbb F_p\), form a partition of \(U_\ell\).
Let
\(
x=(\lambda a,b,c,D,e,\lambda f)\in U_\ell,
\lambda\in\mathbb F_p^\times.
\)
Since \(\lambda\neq 0\), there exists a unique \(d\in\mathbb F_p\) such that \(D=\lambda d\). Also, if we put
\(
t=\lambda^{-1}(ae-fb),
\)
then \(ae-fb=\lambda t\), and hence
\(
x\in L_{\lambda d,\lambda t}(\lambda a,\lambda f)\subseteq G_\ell(d,t).
\)
Thus
\(
U_\ell=\bigcup_{d,t\in\mathbb F_p} G_\ell(d,t).
\)
To see that this union is disjoint, suppose that
\(
x\in G_\ell(d,t)\cap G_\ell(d',t').
\)
Then for some \(\lambda\in\mathbb F_p^\times\),
\(
x\in L_{\lambda d,\lambda t}(\lambda a,\lambda f)\cap L_{\lambda d',\lambda t'}(\lambda a,\lambda f).
\)
Therefore, \(x\) has the fourth coordinate in both \(\lambda d\) and \(\lambda d'\), and also satisfies both equations
\(
ae-fb=\lambda t\text{ and }ae-fb=\lambda t'.
\)
Since \(\lambda\neq 0\), it follows that \(d=d'\) and \(t=t'\). Hence
\[
U_\ell=\bigsqcup_{d,t\in\mathbb F_p} G_\ell(d,t).
\]
Now let
\(
x\in L_{\lambda d,\lambda t}(\lambda a,\lambda f)~ \text{and }
y\in L_{\mu d,\mu t}(\mu a,\mu f),
\)
where \(\lambda,\mu\in\mathbb F_p^\times\). 
Then
$
x=(\lambda a,b,c,\lambda d,e,\lambda f),~
y=(\mu a,B,C,\mu d,E,\mu f),
$
with
\(
ae-fb=\lambda t
\text{ and }
aE-fB=\mu t.
\)
Since \((\mu a,\mu f)=\mu(\lambda^{-1}(\lambda a,\lambda f))\), the cosets \(C_{(\lambda a,\lambda f)}\)~and~ \(C_{(\mu a,\mu f)}\) are linearly dependent. Therefore, the commuting equation \eqref{commuting1} and \eqref{commuting2} is automatic and \eqref{commuting3} becomes $aE-fB=\mu\lambda^{-1}(ae-fb)$. Also,\(
\mu\lambda^{-1}(ae-fb)=\mu\lambda^{-1}(\lambda t)=\mu t=aE-fB.
\)
Hence every two vertices of \(G_\ell(d,t)\) are adjacent, so \(G_\ell(d,t)\) is a clique.

Next, let $
x\in G_\ell(d,t),~ y\in G_\ell(d',t'),$
and suppose that \(x\) and \(y\) commute. Choose \(\lambda,\mu\in\mathbb F_p^\times\) such that
\(
x\in L_{\lambda d,\lambda t}(\lambda a,\lambda f),
y\in L_{\mu d',\mu t'}(\mu a,\mu f).
\)
Again by commuting equations, \(C_{(\lambda a,\lambda f)}\) and \(C_{(\mu a,\mu f)}\), we obtain
\(
\mu d'=\mu d \text{ and }
aE-fB=\mu\lambda^{-1}(ae-fb).
\)
Since \(\mu\neq 0\), the first equality gives \(d'=d\). Also,
\(
\mu t'=aE-fB=\mu\lambda^{-1}(ae-fb)=\mu\lambda^{-1}(\lambda t)=\mu t,
\)
whence \(t'=t\). Therefore distinct sets \(G_\ell(d,t)\) are pairwise nonadjacent.

Finally, for fixed \(\lambda\in\mathbb F_p^\times\), the set
\(
L_{\lambda d,\lambda t}(\lambda a,\lambda f)
\)
is determined inside the coset \(C_{(\lambda a,\lambda f)}\) by the single linear equation
\(
ae-fb=\lambda t.
\)
Since \((a,f)\neq (0,0)\), this equation has exactly \(p\) solutions in \((b,e)\), while \(c\) is arbitrary. Thus
\(
|L_{\lambda d,\lambda t}(\lambda a,\lambda f)|=p^2.
\)
As \(\lambda\) ranges over \(\mathbb F_p^\times\), it follows that
\(
|G_\ell(d,t)|=(p-1)p^2.
\)
Since there are \(p^2\) choices of \((d,t)\in\mathbb F_p^2\)  and \(U_\ell\) is the disjoint union of the pairwise nonadjacent cliques \(G_\ell(d,t)\), we conclude that
$
\Gamma[U_\ell]\cong p^2K_{(p-1)p^2}.
$
\end{proof}

\begin{theorem}
Let \((a,f)\) and \((A,F)\) be linearly independent in \(\mathbb{F}_p^2\). Then
\(
\Gamma[L_0(a,f),L_0(A,F)] \cong pK_{p^2,p^2}.
\)
\end{theorem}
\begin{proof}
Let
\[
L_0(a,f)=\{(a,b,c,0,e,f): b,c,e\in \mathbb{F}_p\}
\] and 
\[
L_0(A,F)=\{(A,B,C,0,E,F): B,C,E\in \mathbb{F}_p\}.
\]
Take
\(
x=(a,b,c,0,e,f)\in L_0(a,f),~ y=(A,B,C,0,E,F)\in L_0(A,F).
\)
From the multiplication rule in \(UT(4,p)\), \(x\) and \(y\) commute if and only if
\(
aE+bF=Ae+Bf,
\)
which is equivalent to
\(
Ae-Fb=aE-fB.
\)
For each \(t\in \mathbb{F}_p\), define
\(
M_t(a,f)=\{(a,b,c,0,e,f)\in L_0(a,f): Ae-Fb=t\}
\)
and
\(
N_t(A,F)=\{(A,B,C,0,E,F)\in L_0(A,F): aE-fB=t\}.
\)
Then
\(
L_0(a,f)=\bigcup_{t\in \mathbb{F}_p} M_t(a,f),
L_0(A,F)=\bigcup_{t\in \mathbb{F}_p} N_t(A,F).
\)
Since \((A,F)\ne (0,0)\), the equation \(Ae-Fb=t\) is a nonzero linear equation in the variables \(b,e\), which has exactly \(p\) solutions in \(\mathbb{F}_p^2\). As \(c\) is free, it follows that
\(
|M_t(a,f)|=p^2.
\)
Similarly,
\(
|N_t(A,F)|=p^2.
\)
A vertex \(x\in L_0(a,f)\) and a vertex \(y\in L_0(A,F)\) commute if and only if
\(
Ae-Fb=aE-fB,
\)
that is, if and only if \(x\) and \(y\) belong to the blocks corresponding to the same parameter \(t\). Hence, every vertex of \(M_t(a,f)\) is adjacent to every vertex of \(N_t(A,F)\), while no vertex of \(M_s(a,f)\) is adjacent to a vertex of \(N_t(A,F)\) whenever \(s\ne t\).
Therefore, for each \(t\in \mathbb{F}_p\),
\(
\Gamma[M_t(a,f),N_t(A,F)]\cong K_{p^2,p^2}.
\)
Since there are \(p\) possible values of \(t\), the bipartite graph between \(L_0(a,f)\) and \(L_0(A,F)\) is the disjoint union of \(p\) complete bipartite graphs \(K_{p^2,p^2}\). Thus
\[
\Gamma[L_0(a,f),L_0(A,F)]\cong pK_{p^2,p^2}.
\]
\end{proof}
The above theorem shows that the interaction between $L_0(a,f)$ and $L_0(A,F)$ is governed by the linear form $Ae-Fb$. Indeed, the commuting condition \eqref{commuting3} naturally partitions $L_0(a,f)$ into $p$ level sets $M_t$. Similarly, $L_0(A,F)$ decomposes into corresponding level sets $N_t$. When additional layers such as $L_0(a_1,f_1)$ are considered, further linear forms arise in the commuting condition. These induce additional partitions of $L_0(a,f)$. The intersections of these partitions refine the structure of $L_0(a,f)$ into smaller subsets, each simultaneously describing the interaction with the corresponding $L_0$-layers of other direction parts.
\begin{theorem}{\label{S1}}
For a direction $\ell=\langle(a,f)\rangle\subset \mathbb{F}_p^2$, define
\[
S_\ell=\{(0,B,C,0,E,0)\in H:aE=fB\} \subset H.
\]
Then the following statements hold.
\begin{enumerate}
\item For every $x\in U_\ell$,
\(
N_H(x)=S_\ell.
\)

\item If $\ell\neq {\ell}'$, then
\(
S_\ell\cap S_{\ell'}=Z(G),
\)
\item
\(
\bigcup_{\ell}S_\ell=\{(0,B,C,0,E,0):B,C,E\in\mathbb{F}_p\}.
\)
\end{enumerate}
\end{theorem}
\begin{proof}
    Let $x=(\lambda a,b,c,d,e,\lambda f)\in U_\ell$ and $h=(0,B,C,D,E,0)\in H$.From the commuting equations, $x$ commutes with $h$ if and only if 
\[
(\lambda a)D=0,~D(\lambda f)=0,~(\lambda a)E=(\lambda f)B.
\]
Since $(a,f)\neq (0,0)$ and $\lambda\neq 0$, the first two equations imply $D=0$. The third equation then reduces to $aE=fB$. Thus, $xh=hx$ if and only if $h\in S_\ell$. Hence $N_H(x)=S_\ell$.
Now suppose $\ell\neq \ell'$ with $\ell=\langle(a,f)\rangle$ and $\ell'=\langle(A,F)\rangle$.
If $h=(0,B,C,0,E,0)\in S_\ell\cap S_{\ell'}$, then
\(
aE=fB,~ AE=FB.
\)
Since $(a,f)$ and $(A,F)$ are linearly independent, the only solution is $B=E=0$.
Thus $h=(0,0,C,0,0,0)\in Z(G)$, proving $S_\ell\cap S_{\ell'}=Z(G)$. 
Finally, every element of the form $(0,B,C,0,E,0)$ satisfies $
aE=fB$ for the direction ${\ell'}=\langle(B,E)\rangle$. Hence it lies in $S_\ell$.
Therefore
\(
\bigcup_{\ell}S_\ell=\{(0,B,C,0,E,0):B,C,E\in\mathbb{F}_p\}.
\)
\end{proof}
\section{Chromatic and Independence numbers}
In this section, we study the chromatic number and independence number of the reduced commuting graph of $UT(4,p)$. We also show that the reduced commuting graph is not perfect.
\begin{proposition}
\(\Gamma_{\mathrm{red}}(G)\) is not perfect.
\end{proposition}

\begin{proof}
It suffices to show that \(\Gamma_{\mathrm{red}}(G)\) contains an induced cycle of length \(5\).
If \(p\) is odd, consider the five noncentral elements
\[(-1,-1,0,0,-1,-1),~(-1,0,0,0,-1,0),~(-1,1,0,0,-1,0),\]
\[(0,-1,0,0,-1,-1),~(1,-1,0,0,0,0).
\]
If \(p=2\), consider instead
\(
(0,0,0,0,0,1),~ (0,0,0,0,1,0),~ (0,1,0,0,0,0),
\)
\(
(1,0,0,0,0,1),~(1,1,0,0,1,0).
\)
A direct computation shows that each of these sets induce a cycle of length \(5\).
Thus \(\Gamma_{\mathrm{red}}(G)\) contains an induced \(5\)-cycle for every prime \(p\). Therefore \(\Gamma_{\mathrm{red}}(G)\) is not perfect.
\end{proof}
\begin{proposition}
The induced subgraph \(\Gamma(G)\setminus \mathcal L_{0}^{*}\) is perfect, where $\mathcal L_0^*
=
\{(a,b,c,0,e,f):(a,f)\in \mathbb F_p^2\setminus\{(0,0)\},\ b,c,e\in\mathbb F_p\}$
\end{proposition}

\begin{proof}
By the decomposition of each direction part \(U_{\ell}\), every vertex of \(\Gamma(G)\setminus (H_{0}\cup \mathcal L_{0}^{*})\) lies in a giant clique \(G_{\ell}(d,t)\) with \(d\neq 0\). Moreover, these cliques are pairwise nonadjacent. Indeed, within a fixed direction part \(U_{\ell}\), distinct giant cliques \(G_{\ell}(d,t)\) are pairwise nonadjacent, and across different direction parts, interaction occurs only through the zero-layers. Hence, after deleting \(\mathcal L_{0}^{*}\), no edges remain between distinct cliques \(G_{\ell}(d,t)\) with \(d\neq 0\). Thus the induced subgraph on \(\Gamma(G)\setminus (H_{0}\cup \mathcal L_{0}^{*})\) is a disjoint union of cliques.
It remains to describe the adjacency with \(H_{0}\). For each direction \(\ell\), every vertex of \(U_{\ell}\) has the same neighborhood in \(H\), namely \(S_{\ell}\). Therefore every clique \(G_{\ell}(d,t)\) with \(d\neq 0\) is joined precisely to the clique \(S_{\ell}^{(0)}:=S_{\ell}\setminus Z(G)\) inside \(H_{0}\). Since \(H_{0}\) itself is a clique, it follows that \(\Gamma(G)\setminus \mathcal L_{0}^{*}\) consists of the clique \(H_{0}\), together with pairwise nonadjacent cliques \(G_{\ell}(d,t)\), each attached along a clique \(S_{\ell}^{(0)}\subseteq H_{0}\).
Consequently, \(\Gamma(G)\setminus \mathcal L_0^{*}\) is obtained by gluing pairwise disjoint cliques to the clique \(H_{0}\) along cliques. Hence it is chordal, and therefore \(\Gamma(G)\setminus \mathcal L_{0}^{*}\) is perfect.
\end{proof}
\begin{theorem}
    $\chi(\Gamma_{red}(G))=p^4-p$
\end{theorem}
\begin{proof}
Let
$
G_0=\Gamma_{red}(G),~ H_0=H\setminus Z(G) \text{ and }S_\ell^{(0)}=S_{\ell} \setminus Z(G)$. Since $H_0$ is a clique of $p^4-p$ vertices, $\chi(G_0) \geq p^4-p$. Therefore, if one can show that the colors already used in $H_0$ can be redistributed to $G_0 \setminus H_0$ so as to obtain a proper coloring of $G_0$, then $\chi(G_0)=p^4-p$..
By Theorem \ref{giant_clique},
$\Gamma[U_\ell]\cong p^2K_{(p-1)p^2}.$ That is, $\Gamma[U_\ell]$ is a disjoint union of $p^2$ cliques $K_{(p-1)p^2}$.
Hence each $U_\ell$ can be properly colored using $(p-1)p^2$ colors. There are $p+1$ such $U_{\ell}$. 
However, by Theorem \ref{L0}, the layers $L_0(a,f)$ of $U_\ell$ interact with the layers $L_0(A,F)$ of $U_{\ell'}$. Therefore, at this stage $(p+1)(p-1)p^2=p^4-p^2$ colors have been assigned within $G_0\setminus H_0$. Now, the following arguments show how the redistribution of $p^4-p$ colors is carried out on $G_0 \setminus H_0$.
let $R = H_0 \setminus \bigcup_{\ell} S_\ell^{(0)}$. Then
$|R| = (p^4 - p) - (p^3 - p)=p^4-p^3$.
By Theorem \ref{S1}, colors coming from $S_\ell^{(0)}$ are forbidden only for $U_\ell$. Use exactly $|S_{\ell-1}^{(0)}| = p^2 - p$ colors of $S_{\ell-1}^{(0)}$ (cyclic indices mod $p+1$) to color each $U_\ell$ in $G_0 \setminus H_0$. Total number of colors used for coloring $G_0 \setminus H_0$ is $(p+1)(p^2 - p) = p^3 - p$.
Remaining number of colors needed are $p^4 - p^2 - (p^3 - p)$.
Now for all these, we use colors from $R$. This is possible because we have enough colors in $R$ since $|R|=p^3(p-1) > p^4 - p^3 - p^2 + p$. Therefore, only $p^4-p$ colors from $H_0$ are enough to color the whole graph.
\end{proof}
The following theorem expresses the independence number of $UT(4,p)$ in terms of the independence number of the induced subgraph on \(
\mathcal L_0^*
=
\{(a,b,c,0,e,f):(a,f)\in \mathbb F_p^2\setminus\{(0,0)\},\ b,c,e\in\mathbb F_p\},
\)
\begin{theorem}
\(
\alpha(\Gamma(G))=p^3-p+1+\alpha\!\bigl(\Gamma[\mathcal L_0^*]\bigr).
\)
\end{theorem}

\begin{proof}
By Theorem \ref{giant_clique}, for each direction part \(U_\ell\), there are \(p^2-p\) giant cliques of the form \(G_\ell(d,t)\) with \(d\neq 0\). Since there are \(p+1\) direction parts, the total number of such giant cliques is
\(
(p+1)(p^2-p)=p^3-p.
\)
By Theorem \ref{L0}, these giant cliques are pairwise nonadjacent across different direction parts, and each of them is a clique. Hence an independent set may contain exactly one vertex from each of them, but no more. Therefore their total contribution to the independence number is exactly
\(
p^3-p.
\)
The remaining \(p\) giant cliques in each \(U_\ell\) together form \(\mathcal L_0^*\), so their contribution is precisely
\(
\alpha\!\bigl(\Gamma[\mathcal L_0^*]\bigr).
\)
Finally, by Theorem \ref{S1}, there are vertices in $H$ which are nonadjacent to all vertices chosen from the \(p^3-p\) giant cliques and also nonadjacent to every vertex in \(\mathcal L_0^*\). Since \(H\) is a clique, at most one vertex from \(H\) can belong to an independent set. Hence the contribution from \(H\) is exactly \(1\).
Combining these three contributions, we obtain
\(
\alpha(\Gamma(G))=p^3-p+1+\alpha\!\bigl(\Gamma[\mathcal L_0^*]\bigr).
\)
This completes the proof.
\end{proof}
To understand the structure of the induced commuting graph on the zero-layer union $\mathcal L_0^*$, it is natural to look for large cliques that are compatible with the commuting relation. For two vertices
\(
x=(a,b,c,0,e,f),~ y=(A,B,C,0,E,F)\in \mathcal L_0^*,
\)
the commuting condition is \eqref{commuting3}.
Thus, if one can parametrize \(b\) and \(e\) linearly in terms of \(a\) and \(f\) in such a way that the two sides of \eqref{commuting3} become identical for all choices of \((a,f)\) and \((A,F)\), then one obtains a clique.
This suggests seeking constants \(\alpha,\beta,\gamma,\delta\in \mathbb F_p\) such that
\(
b=\alpha a+\beta f,~ e=\gamma a+\delta f
\). If we impose the same relations on the second vertex, $B=\alpha A+\beta F,~E=\gamma A+\delta F$ and substituting these expressions into \eqref{commuting3}, a direct calculation shows that $\delta=-\alpha$. Writing $\gamma= \eta \beta$ for a fixed parameter $\eta \in \mathbb F_p $, we are led to the family,
\[
K_{\alpha,\beta}
=
\left\{
(a,\alpha a+\beta f,c,0,\eta\beta a-\alpha f,f)
:
(a,f)\in \mathbb F_p^2\setminus\{(0,0)\},\ c\in \mathbb F_p
\right\},
\]
where \((\alpha,\beta)\in\mathbb F_p^2\).
The characteristics of this family can be understood from the following theorem.
\begin{theorem}{\label{clique_decomposition}}
Let \(p\) be an odd prime, and choose \(\eta\in \mathbb F_p\) such that \(-\eta\) is a nonsquare in \(\mathbb F_p\). For each \((\alpha,\beta)\in\mathbb F_p^2\), define
\[
K_{\alpha,\beta}
=
\left\{
(a,\alpha a+\beta f,c,0,\eta\beta a-\alpha f,f)
:
(a,f)\in \mathbb F_p^2\setminus\{(0,0)\},\ c\in \mathbb F_p
\right\}.
\]
Then:
\begin{enumerate}
    \item each \(K_{\alpha,\beta}\) is a clique in \(\Gamma[\mathcal L_0^*]\);
    \item the family \(\{K_{\alpha,\beta}:(\alpha,\beta)\in\mathbb F_p^2\}\) is pairwise disjoint;
    \item these cliques cover \(\mathcal L_0^*\).
\end{enumerate}
Hence
\(
\mathcal L_0^*
=
\bigsqcup_{(\alpha,\beta)\in\mathbb F_p^2} K_{\alpha,\beta}.
\)
Moreover, each clique has size
\(
|K_{\alpha,\beta}|=p^3-p.
\)
\end{theorem}

\begin{proof}
We prove the three assertions in turn.

Take two arbitrary vertices of \(K_{\alpha,\beta}\):
\(
x=(a,\alpha a+\beta f,c,0,\eta\beta a-\alpha f,f)
\) and 
\(
y=(A,\alpha A+\beta F,C,0,\eta\beta A-\alpha F,F).
\)
For vertices of \(\mathcal L_0^*\), the commuting condition \eqref{commuting3} applies. Here
\(
b=\alpha a+\beta f,~ e=\eta\beta a-\alpha f,
\) and
\(
B=\alpha A+\beta F,~ E=\eta\beta A-\alpha F.
\)
Now
\[
\begin{aligned}
aE+bF
=a(\eta\beta A-\alpha F)+(\alpha a+\beta f)F
&=\eta\beta aA-\alpha aF+\alpha aF+\beta fF \\
&=\eta\beta aA+\beta fF.
\end{aligned}
\]
Similarly,
\[
\begin{aligned}
Ae+Bf
=A(\eta\beta a-\alpha f)+(\alpha A+\beta F)f
&=\eta\beta aA-\alpha Af+\alpha Af+\beta Ff \\
&=\eta\beta aA+\beta fF.
\end{aligned}
\]
Thus \eqref{commuting3} is satisfied and 
hence \(x\) and \(y\) commute. Hence any two vertices of \(K_{\alpha,\beta}\) are adjacent and \(K_{\alpha,\beta}\) is a clique.

Take an arbitrary vertex
\(
x=(a,b,c,0,e,f)\in\mathcal L_0^*.
\)
So \((a,f)\neq(0,0)\). We seek \(\alpha,\beta\in\mathbb F_p\) such that
\(
b=\alpha a+\beta f, ~e=\eta\beta a-\alpha f. 
\)
This is the linear system
\[
\begin{pmatrix}
a & f\\
-f & \eta a
\end{pmatrix}
\binom{\alpha}{\beta}
=
\binom{b}{e}.
\]
Its determinant is
\(
\eta a^2+f^2.
\)
If \(a=0\), then \(f\neq 0\), so \(\eta a^2+f^2=f^2\neq 0\). If \(a\neq 0\) and \(\eta a^2+f^2=0\), then
\(
\left(\frac{f}{a}\right)^2=-\eta,
\)
contradicting the assumption that \(-\eta\) is a nonsquare. Hence \(\eta a^2+f^2\neq 0\) for all \((a,f)\neq(0,0)\).
Therefore the system has a unique solution \((\alpha,\beta)\), and hence \(x\in K_{\alpha,\beta}\). Since \(x\) was arbitrary, the union of all \(K_{\alpha,\beta}\) is \(\mathcal L_0^*\).

Suppose
\(
x\in K_{\alpha,\beta}\cap K_{\alpha',\beta'}.
\)
Then there exist \((a,f)\neq(0,0)\), \((a',f')\neq(0,0)\), and \(c,c'\in\mathbb F_p\) such that
\(
x=(a,\alpha a+\beta f,c,0,\eta\beta a-\alpha f,f),
\)
and
\(
x=(a',\alpha' a'+\beta' f',c',0,\eta\beta' a'-\alpha' f',f').
\)
Equality of the first, third, and sixth coordinates gives
\(
a=a',~c=c',~ f=f'.
\)
Comparing the second and fifth coordinates, we obtain
\(
\alpha a+\beta f=\alpha' a+\beta' f, 
\)
\(
\eta\beta a-\alpha f=\eta\beta' a-\alpha' f.
\)
which implies
\[
(\alpha-\alpha')a+(\beta-\beta')f=0,
 \text{ and }
\eta(\beta-\beta')a-(\alpha-\alpha')f=0.
\]
Thus \((a,f)\neq(0,0)\) is a nonzero solution of the homogeneous system
\[
\begin{pmatrix}
\alpha-\alpha' & \beta-\beta'\\
\eta(\beta-\beta') & -(\alpha-\alpha')
\end{pmatrix}
\binom{a}{f}
=
\binom00.
\]
Hence the determinant must vanish. That is,
\[
-(\alpha-\alpha')^2-\eta(\beta-\beta')^2=0 \implies 
(\alpha-\alpha')^2+\eta(\beta-\beta')^2=0.
\] If \(\beta-\beta'\neq 0\), then dividing  by \((\beta-\beta')^2\) yields
\(
\left(\frac{\alpha-\alpha'}{\beta-\beta'}\right)^2=-\eta,
\)
again contradicting that \(-\eta\) is a nonsquare. Therefore \(\alpha=\alpha'\) and \(\beta=\beta'\). So distinct members of the family are disjoint.

Finally, for fixed \((\alpha,\beta)\), the set \(K_{\alpha,\beta}\) is determined by
\ \(p^2-1\) choices of \((a,f)\neq(0,0)\) \text{ and } \(p\) choices of \(c\).
Hence
\(
|K_{\alpha,\beta}|=p(p^2-1)=p^3-p.
\)
This completes the proof.
\end{proof}

\begin{corollary}{\label{lower}}
\(
\alpha\bigl(\Gamma[\mathcal L_0^*]\bigr)\le p^2.
\)
\end{corollary}

\begin{proof}
By Theorem \ref{clique_decomposition},
$\mathcal L_0^*$ is a disjoint union of cliques $K_{\alpha, \beta}$. Hence any independent set in \(\Gamma[\mathcal L_0^*]\) meets each \(K_{\alpha,\beta}\) in at most one vertex. Since there are exactly \(p^2\) such cliques, it follows that
\(
\alpha\bigl(\Gamma[\mathcal L_0^*]\bigr)\le p^2.
\)
\end{proof}
\begin{proposition}{\label{upper}}
\(
\alpha\bigl(\Gamma[\mathcal{L}_0^*]\bigr)\geq 3p-2.
\)
\end{proposition}

\begin{proof}
Consider the subsets
\[
X=\{(a,1,0,0,1,0):a\in\mathbb{F}_p^\times\},
Y=\{(0,1,0,0,0,f):f\in\mathbb{F}_p^\times\},\]
\[Z=\{(1,0,0,0,t,1):t\in\mathbb{F}_p^\times\},
\]
and choose
\(
w=(u,0,0,0,0,v), u,v\in\mathbb{F}_p^\times,u\neq v.
\)
Set
\(
I=X\cup Y\cup Z\cup\{w\}.
\)
By direct substitution into \eqref{commuting3}, each of \(X\), \(Y\), and \(Z\) is an independent set. Now let
\(
x=(a,1,0,0,1,0)\in X,~y=(0,1,0,0,0,f)\in Y,~ z=(1,0,0,0,t,1)\in Z.
\)
Then \eqref{commuting3} for \(x\) and \(y\) gives \(f=0\), which is impossible. For \(x\) and \(z\), it gives \(at=0\), impossible since \(a,t\in\mathbb{F}_p^\times\). For \(y\) and \(z\), it gives \(1=0\), again impossible. Hence there are no edges between the distinct sets \(X\), \(Y\), and \(Z\). Finally, \eqref{commuting3} for \(w\) and a vertex of \(X\) gives \(u=v\), impossible by choice of \(w\). With a vertex of \(Y\), it gives \(v=0\), impossible, and with a vertex of \(Z\), it gives \(ut=0\), impossible. Thus \(w\) is nonadjacent to every vertex of \(X\cup Y\cup Z\). Therefore \(I\) is an independent set in \(\Gamma[\mathcal{L}_0^*]\). Since
\(
|I|=|X|+|Y|+|Z|+1=(p-1)+(p-1)+(p-1)+1=3p-2,
\)
we obtain
\(
\alpha\bigl(\Gamma[\mathcal{L}_0^*]\bigr)\geq 3p-2.
\)
\end{proof}
\section{Conclusion}
Since
\(
\alpha(\Gamma(G))=p^3-p+1+\alpha\bigl(\Gamma[L_0^*]\bigr),
\)
Corollary \ref{lower} and Proposition \ref{upper} yield
\(
p^3+2p-1\le \alpha(\Gamma(G))\le p^3+p^2-p+1.
\)
Thus the determination of the independence number of the reduced commuting graph is reduced to the determination of \(\alpha\bigl(\Gamma[\mathcal{L}_0^*]\bigr)\). In particular, further progress depends on a sharper analysis of the interaction between the cliques \(K_{\alpha,\beta}\) inside \(\mathcal{L}_0^*\), or equivalently on the associated reduced block graph.

%================================================

\bibliographystyle{elsarticle-num}
\bibliography{references}
\end{document}